\documentclass{ws-rmta}

\usepackage{amsmath}
\usepackage{amstext}
\usepackage{amsfonts}
\usepackage{amssymb}
\usepackage{graphicx}
\usepackage{bbm}

\usepackage{tikz}
\usepackage{pgfplots}
\usepackage[T2A,T1]{fontenc}

\DeclareMathOperator{\tr}{Tr}

\def\restriction#1#2{\mathchoice
              {\setbox1\hbox{${\displaystyle #1}_{\scriptstyle #2}$}
              \restrictionaux{#1}{#2}}
              {\setbox1\hbox{${\textstyle #1}_{\scriptstyle #2}$}
              \restrictionaux{#1}{#2}}
              {\setbox1\hbox{${\scriptstyle #1}_{\scriptscriptstyle #2}$}
              \restrictionaux{#1}{#2}}
              {\setbox1\hbox{${\scriptscriptstyle #1}_{\scriptscriptstyle #2}$}
              \restrictionaux{#1}{#2}}}
\def\restrictionaux#1#2{{#1\,\smash{\vrule height .8\ht1 depth .85\dp1}}_{\,#2}} 

\newcounter{ccorollary}
\newtheorem{corollaire}[ccorollary]{Corollary}

\pgfplotsset{compat=1.11}

\begin{document}

%
\catchline{}{}{}{}{}
%

\title{Sharp Bounds for the Concentration of the Resolvent \\ in Convex Concentration Settings}

\author{Cosme Louart\footnote{
GIPSA-lab.}}

\address{GIPSA-lab, 11 rue des Mathématiques,
38401 St Martin d'Hères\\ 
\email{cosmelouart@gmail.com} }

\maketitle

\begin{history}
\received{(26/12/2021)}
\end{history}

\begin{abstract}
Considering random matrix $X \in \mathcal M_{p,n}$ with independent columns satisfying the convex concentration properties issued from a famous theorem of Talagrand, we express the linear concentration of the resolvent $Q = (I_p - \frac{1}{n}XX^T) ^{-1}$ around a classical deterministic equivalent with a good observable diameter for the nuclear norm. The general proof relies on a decomposition of the resolvent as a series of powers of $X$.
\end{abstract}



\section*{Introduction}
Initiated by Milman in the 70s \cite{MIL71}, the theory of concentration of the measure provided a wide range of concentration inequalities, mainly concerning continuous distribution (i.e. with no atoms), in particular thanks to the beautiful interpretation with the bound on the Ricci curvature \cite{GRO79}. To give a simple fundamental example (more examples can be found in \cite{LED05}), a random vector $Z \in \mathbb R^n$ having independent and Gaussian entries with unit variance satisfies for any $1$-Lipschitz mapping $f: \mathbb R^n \to \mathbb R$:
\begin{align}\label{eq:concentration_lipschitz}
  \forall t >0: \ \ \mathbb P \left( \left\vert f(Z) - \mathbb E[f(Z)] \right\vert \geq t \right) \leq 2 e^{-t^2/2}.
\end{align}
This inequality is powerful for two reasons: first, it is independent on the dimension $n$, second it concerns any $1$-Lipschitz mapping $f$. It is then interesting to formalize this behavior to introduce a class of ``Lipschitz concentrated random vectors'' satisfying the same concentration inequality as the Gaussian distribution (in particular, all the Lipschitz transformation of a Gaussian vector). This was done in several books and papers where this approach proved its efficiency in the study of random matrices \cite{TAO12,ELK09,LOU21RHL}...

We want here to extend those results to a new class of concentrated vectors discovered by  Talagrand in \cite{TAL95}. Although the concentration result looks similar, its nature is quite different as it concerns bounded distributions for which classical tools of differential geometry do not operate. In a sense, it could be seen as a combinatorial result of concentration.
Given a random vector $Z \in [0,1]^n$ with independent entries, this result sets that for any $1$-Lipschitz \emph{and convex} mapping $f: [0,1] \to \mathbb R$:
\begin{align}\label{eq:concentration_convexe}
  \forall t >0: \ \ \mathbb P \left( \left\vert f(Z) - \mathbb E[f(Z)] \right\vert \geq t \right) \leq 2 e^{-t^2/4}.
\end{align}
We can mention here the recent results of \cite{Hua21} that extend this kind of inequalities for random vectors with independent and subgaussian entries.
Adopting the terminology of \cite{VU2014,MEC11,ADA11}, we call those vectors convexly concentrated random vector (see Definition~\ref{def:cnocentration_convexe} below).
The convexity required for the observations to concentrate makes the discussion on convexly concentrated random vector far more delicate. There is no more stability towards Lipschitz transformations and given a convexly concentrated random vector $Z$, just its affine transformations are sure to be concentrated. 
This issue raises naturally for one of the major objects of random matrix theory, namely the resolvent $Q^z = (zI_n - X)^{-1}$ that can provide important eigen properties on $X$. 
In the case of convex concentration, the concentration of the resolvent $Q^z = (zI_n - X)^{-1}$ is no more a mere consequence of a bound on its differential on $X\in \mathcal{M}_{p}$. Still, as first shown by \cite{Gui00}, it is possible to obtain concentration properties on the sum of Lipschitz functionals of the eigen values. Here we pursue the study, looking at \textit{linear} concentration properties of $Q^z$ for which similar inequalities to \eqref{eq:concentration_lipschitz} or~\eqref{eq:concentration_convexe} are only satisfied by $1$-Lipschitz \emph{and linear} functionals $f$. 
The well known identity
\begin{align}\label{eq:importance_stieltjes_transform}
  \frac{1}{p}\sum_{\lambda \in \text{Sp}(X)} f(\lambda) = - \frac{1}{2i\pi} \oint_\gamma f(z) \tr(Q^z)dz,
\end{align}
is true for any analytical mapping $f$ defined on the interior of a path $\gamma \in \mathbb C$ containing the spectrum of $X$ (or any limit of such mappings), therefore, our results on the concentration of $Q^z$ concern in particular the quantities studied in \cite{Gui00}.


Although it is weaker\footnote{Lipschitz concentrated random vectors are convexly and linearly concentrated, convexly concentrated random vectors are linearly concentrated.}, the class of linearly concentrated vectors behaves very well towards the dependence and the sum and allows us to obtain the concentration of the resolvent expressing it as a sum $Q^z = \frac{1}{z}\sum_{i=1}^\infty (X/z)^i$. The linear concentration of the powers of $X$ was justified in\footnote{We provide an alternative proof in the appendix.} \cite{MEC11} in the case of convexly concentrated random matrix $X$. We call this weakening of the concentration property ``the degeneracy of the convex concentration through multiplication''. The linear concentration of the resolvent is though sufficient for most practical applications that rely on an estimation of the Stieltjes transform $m(z) = \frac{1}{p} \tr(Q^z)$.

We present below our main contribution without the useful but non-standard formalism introduced in the rest of the article. It concerns the concentration and the estimation of\footnote{The concentration of $ (zI_p - \frac{1}{\sqrt n}X)^{-1}$ for a positive symmetric matrix $X \in \mathcal{M}_{p}$ is immediate from our approach. The estimation of its expectation is more laborious and goes beyond the scope of the paper although it can be obtained with the same tools.}
\begin{align*}
   Q^z \equiv \left( zI_p - \frac{1}{n}XX^T \right)^{-1}
 \end{align*} for a random matrix $X \in \mathcal{M}_{p,n}$.
Following the formalism of the random matrix theory, the computable estimation of $\mathbb E[Q^z]$ will be called a ``deterministic equivalent''.
Its definition relies on a well known result that states that given a family of symmetric matrices $\Sigma = (\Sigma_1,\ldots, \Sigma_n) \in \mathcal{M}_{p}^n$, there exists a unique vector $\tilde \Lambda_\Sigma^z \in \mathbb C^n$ satisfying:
  \begin{align*}
    \forall i\in [n]:&
    &[\tilde \Lambda_\Sigma^z]_i = \frac{1}{n} \tr \left( \Sigma_i \tilde Q_\Sigma^{\tilde \Lambda_\Sigma^z} \right)&
    &\text{with } \ \tilde Q_{\Sigma}^{\tilde \Lambda_\Sigma^z} = \left( zI_p - \frac{1}{n}  \sum_{j=1}^n \frac{\Sigma_j}{1 - [\tilde \Lambda_\Sigma^z]_j} \right)^{-1}.
  \end{align*}

With those notations at hand, let us state:
\begin{theorem}[Concentration of the resolvent]\label{the:concentration_resolvent_main}
  Considering two sequences $(p_n)_{n\in \mathbb N} \in \mathbb N^{\mathbb N}$, $(\sigma_n) \in \mathbb R_+^{\mathbb N}$ and four constants $c,C,K, \gamma >0$, we suppose that we are given for any $n\in \mathbb N$ a random matrix: $X_n = (x_1^{(n)}, \ldots, x_n^{(n)})\in \mathcal M_{p_n,n}$ such that 
  \begin{itemize}
    \item $p_n \leq \gamma n$
    \item for all $n \in \mathbb N$, $x_n^{(1)}, \ldots, x_n^{(n)}$ are independent,
    \item $\sup_{n \in \mathbb N,j\in [n]} \left\Vert \mathbb E \left[ x_n^{(j)}  \right] \right\Vert \leq K \sqrt n$
    \item for any $n \in \mathbb N$ quasi-convex mapping $g: \mathcal M_{n,p_n}^{m_n} \to \mathbb R$, $1$-Lipschitz for the euclidean norm:
  \begin{align*}
    \mathbb P \left( \left\vert g(X_n) - \mathbb E \left[ g(X_n) \right] \right\vert \geq t\right) \leq C e^{-c(t/\sigma_n)^2}.
  \end{align*}
   \end{itemize} 
  Then for any constant $\varepsilon>0$, there exist two constants $c',C'>0$ such that for all $n \in \mathbb N$, for any deterministic matrix $A \in \mathbb R^n$ such that
  $\|A\| \leq 1$ and for any $z \in \mathbb C$, such that $d(z, \text{Sp}(\frac{1}{n}XX^T)) \geq \varepsilon$:
  \begin{align*}
    \mathbb P \left( \left\vert  \tr \left( \left(Q^z - \tilde Q_\Sigma^{\tilde \Lambda_\Sigma^z} \right) A \right) \right\vert \geq t\right) \leq C' e^{ -c' (t/\sigma_n)^2 }  + C' e^{-c'n},
  \end{align*}
    where $\Sigma_i = \mathbb E[x_ix_i^T]$. Besides there exists a constant $K>0$ such that for all $n \in \mathbb N$:
    \begin{align*}
      \left\Vert \mathbb E [Q^z ] - \tilde Q_\Sigma^{\tilde \Lambda_\Sigma^z} \right\Vert_F \leq \frac{K}{\sqrt n}.
    \end{align*}
\end{theorem}

This theorem allows us to get good inferences on the eigen values distribution through the identity \eqref{eq:importance_stieltjes_transform} and the estimation of the Stieltjes transform $g(z) \equiv -\frac{1}{p} \tr (Q^z) $ satisfying the concentration inequality:
\begin{align*}
  \mathbb P \left( \left\vert g(z) + \frac{1}{n} \tr \left( \tilde Q^{\tilde \Lambda^z} \right) \right\vert \geq t \right)\leq Ce^{-c p^2 t^2},
\end{align*}
for two constants $C,c>0$ (and for $d(z, \text{Sp}(\frac{1}{n}XX^T)) \geq O(1)$).

When the distribution of the spectrum of $\frac{1}{n}XX^T$ presents different bulks, this theorem also allows us to understand the eigen-spaces associated to those different bulks. Indeed, considering a path $\gamma \in \mathbb C$ containing a bulk of eigen-values $B \subset \text{Sp}(\frac{1}{n}XX^T)$, if we note $E_B$ the associated random eigen-space and $\Pi_B$ the orthogonal projector on $E_B$, then for any deterministic matrix $A \in \mathcal M_{p}$:
 \begin{align}\label{eq:estimation_eigenspaces}
   \tr(\Pi_BA) = -\frac{1}{2i\pi}\int_\gamma \tr(AQ^z) dz&
 \end{align}
  we can estimate this projection $\Pi_B$ defining $E_B$ thanks to the concentration inequality\footnote{for the concentration to be valid on all the values of the path $\gamma$, one must be careful to consider a path staying at a distance $O(1)$ from the bulk, that is why we only consider here multiple bulk distributions}:
\begin{align*}
  \forall t >0 : \ \ \mathbb P \left( \left\vert \frac{1}{\text{Rg}(\Pi)} \tr(\Pi Q^z) - \frac{1}{\text{Rg}(\Pi)} \tr \left( \Pi\tilde Q^{\tilde \Lambda^z} \right) \right\vert \geq t \right) \leq Ce^{-c\text{Rg}(\Pi)^2t^2},
\end{align*}
for some constants $C,c>0$ and for any projector $\Pi$ defined on $\mathbb R^p$.

The approach we present here does not only allows us to set the concentration of $Q^z$, but also the concentration of any polynomial of finite degree taking as variable combination of $Q^z$, $X$ and $X^T$. The general idea is to develop the polynomial as an infinite series of powers of $X$ in a way that the observable diameters of the different terms of the series sum to the smallest value possible. As it is described in the proof of Proposition~\ref{pro:Concentration_lineaire_Q_proche_spectre}, the summation becomes slightly elaborate when $z$ gets close to the spectrum.

After presenting the definition and the basic properties of the convex and linear concentration (\textbf{Section 1}), we express the concentration of the sum of linearly concentrated random vectors (\textbf{Section 2}). Then we express the concentration of the entry wise product and the matricial product of convexly concentrated random vectors and matrices (\textbf{Section 3}). Finally we deduce the concentration of the resolvent and provide a computable deterministic equivalent (\textbf{Section 4}).

\section{Definition and first properties}

The concentration inequality \eqref{eq:concentration_convexe} is actually also valid for quasi-convex functionals defined folowingly.
\begin{definition}\label{def:quasi_convexe}
  Given a normed vector space $(E,\|\cdot \|)$, an application $f : E \mapsto \mathbb R$ is said to be quasi-convex iif for any $t \in \mathbb R$, the set $\{f \leq t\} \equiv \{x \in E \ | \ f(x) \leq t\}$ is convex.
\end{definition}
The theory of concentration of measure becomes relevant only when dimensions get big. In the cases under study in this paper, the dimension is either given by the number of entries, either by the number of columns $n$ of random matrices - the number of rows $p$ is then understood to depend on $n$, we will sometimes note $p = p_n $. 
We follow then the approach with Levy families \cite{Le51} whose aim is to track the concentration speed through dimensionality. Therefore, we do not talk about a static concentration of a vector but about the concentration of \emph{a sequence of random vectors} as seen in the definition below. In this paper, $E_n$ will either be $\mathbb R^{n}$, $\mathbb R^{p_n}$ $\mathcal{M}_{n}$, $\mathcal{M}_{p_n}$ or $\mathcal{M}_{p_n,n}$. 

There will generally be three possibilities for the norms defining the Lipschitz character of the concentrated observations. Talagrand Theorem gives the concentration for the euclidean norm - i.e. the Frobenius norm for matrices - but we will see that some concentrations are expressed with the nuclear norm (the dual norm of the spectral norm). Given two integers $l,m \in \mathbb N$, the euclidean norm on $\mathbb R^l$ is noted $\|\cdot\|$, the spectral, Frobenius and nuclear norm are respectively defined for any $M \in\mathcal M_{l,m}$ with the expressions:
\begin{align*}
  \|M\| = \sup_{x \in \mathbb R^m} \|Mx\|;&
  &\|M\|_F = \sqrt{\tr(MM^T)};&
  &\|M\|_* = \tr \left( \sqrt{MM^T}  \right).
\end{align*}
\begin{definition}\label{def:cnocentration_convexe}
  \sloppypar{Given a sequence of normed vector spaces $(E_n, \Vert \cdot \Vert_n)_{n\geq 0}$, a sequence of random vectors $(Z_n)_{n\geq 0} \in \prod_{n\geq 0} E_n$, a sequence of positive reals $(\sigma_n)_{n\geq 0} \in \mathbb R_+ ^{\mathbb N}$, we say that $Z =(Z_n)_{n\geq 1}$ is \emph{convexly  concentrated} with an \emph{observable diameter} of order $O(\sigma_n)$ iff there exist two positive constants $C,c>0$ such that $\forall n \in \mathbb N$ and for any $1$-Lipschitz and quasi-convex function $f : E_n \rightarrow \mathbb{R}$ (for the norms $\Vert \cdot \Vert_n$)\footnote{In this inequality, one could have replaced the term ``$\mathbb E[f(Z_n)]$'' by ``$f(Z_n')$'' (with $Z_n'$, an independent copy of $Z_n$) or by ``$m_{f}$'' (with $m_{f}$ a median of $f(Z_n)$). All those three definitions are equivalent.}, }
     \begin{align*}
     \forall t>0:&
     &\mathbb P \left(\left\vert f(Z_n) - \mathbb E[f(Z_n)]\right\vert\geq t\right) \leq C e^{-c(t/\sigma_n)^2},
     \end{align*}
   We write in that case\footnote{The index $2$ in ``$\mathcal E_2$'' is here a reference to the power of $t$ in the concentration bound $C e^{-c(t/\sigma_n)^2}$, we will see some example where this exponent is $1$, in particular in the Hanson-Wright Theorem~\ref{the:hanson_wright} where we will let appear a notation ``$\mathcal E_1$''.} $Z_n \propto_c \mathcal E_{2}(\sigma_n)$ (or more simply $Z \propto_c \mathcal E_{2}(\sigma)$).
\end{definition}
The Theorem of Talagrand then writes:
\begin{theorem}[\cite{TAL95}]\label{the:talagrand}
  A (sequence of) random vector $Z \in [0,1]^n$ with independent entries satisfies $Z \propto_c \mathcal E_2$.
\end{theorem}
Convex concentration is preserved through affine transformations (as for the class of linearly concentrated vectors). Given two vector spaces, $E$ and $F$, we note $\mathcal A(E,F)$ the set of affine transformation from $E$ to $F$, and given $\phi\in \mathcal A(E,F)$, we decompose $\phi = \mathcal L(\phi) + \phi(0)$, where $\mathcal L(\phi)$ is the linear part of $\phi$ and $\phi(0)$ is the translation part. When $E=F$, $\mathcal A(E,F)$ is simply noted $\mathcal A(E)$.
\begin{proposition}\label{pro:stabilite_concentration_convexe_affine}
  Given two normed vector spaces $(E,\|\cdot\|)$ and $(F,\|\cdot\|')$, a random vector $Z \in E$ and an affine mapping $\phi \in \mathcal A(E,F)$ such that $\|\mathcal L(\phi)\|\leq \lambda$:
  \begin{align*}
    Z \propto_c \mathcal E_2(\sigma)&
    &\Longrightarrow&
    &\phi(Z) \propto_c \mathcal E_2(\lambda\sigma).
  \end{align*}
\end{proposition}

We pursue our presentation with the introduction of the linear concentration. It is the ``minimal'' hypothesis necessary on a random vector $X$ to be able to bound quantities of the form $\mathbb E[\|X - \mathbb E[X]\|]$, as it has been explained in \cite{LOU19}. Here we will need its stability towards the sum when we will express $Q^z$ as an infinite series.
\begin{definition}[Linearly concentrated vectors]\label{def:linear_concentration}
  Given a sequence of normed vector spaces $(E_n, \Vert \cdot \Vert_n)_{n\geq 0}$, a sequence of random vectors $(Z_n)_{n\geq 0} \in \prod_{n\geq 0} E_n$, a sequence of deterministic vectors $(\tilde Z_n)_{n\geq 0} \in \prod_{n\geq 0} E_n$, a sequence of positive reals $(\sigma_n)_{n\geq 0} \in \mathbb R_+ ^{\mathbb N}$, $Z_n$ is said to be \emph{linearly concentrated} around the \emph{deterministic equivalent} 
  $\tilde Z_n$ with an \emph{observable diameter} of order $O(\sigma_n)$ iff there exist two constants $c,C >0$ such that $\forall n \in \mathbb N$ and for any unit-normed linear form $f \in E_n'$ ($\forall n \in \mathbb N$, $\forall x \in E$: $|f(x)| \leq \|x\|_n$):
  \begin{align*}
    \forall t>0:& 
      &\mathbb P \left(\left\vert f(Z_n ) - f(\tilde Z_n)\right\vert\geq t\right) \leq C e^{-c(t/\sigma_n)^2}.
  \end{align*}
  When the property holds, we write $Z \in \tilde Z \pm \mathcal E_2(\sigma)$. 
  If it is unnecessary to mention the deterministic equivalent, we will simply write $Z \in \mathcal E_2(\sigma)$; and if we just need to control the order of the norm of the deterministic equivalent, we can write $Z \in O(\theta)\pm \mathcal E_2(\sigma)$ when $\|\tilde Z_n \|_n \leq O(\theta_n)$.
\end{definition}
In the literature \cite{bou13}, those vectors are commonly called sub-Gaussian random vectors. 

The notions of linear concentration, convex concentration (and Lipschitz concentration) are equivalent for random variables and we have this important characterization with the moments:
\begin{proposition}[\cite{LED05}, Proposition 1.8., \cite{LOU19}, Lemma 1.22.]\label{pro:concentration_variable_caracterisation_moment}
  Given a sequence of random variables $Z_n \in \mathbb R$ and a sequence of positive parameters $\sigma_n>0$, we have the equivalence:
  \begin{align*}
    Z_n \propto_c \mathcal E_2(\sigma_n)
    &\Longleftrightarrow
    Z_n \in \mathbb E[Z_n] \pm \mathcal E_2(\sigma_n)\\
    &\Longleftrightarrow \exists C >0 \ | \ \forall n,m \in \mathbb N:
    \mathbb E \left[ \left\vert Z_n - \mathbb E[Z_n] \right\vert^n \right] \leq C m^{\frac{m}{2}}\sigma_n^m \\
    &\Longleftrightarrow \exists C >0 \ | \ \forall n \in \mathbb N, \forall r>0: \mathbb E \left[ \left\vert Z_n - \mathbb E[Z_n] \right\vert^r \right] \leq Cr^{\frac{r}{2}}\sigma_n^r .
  \end{align*}
\end{proposition}

We end with a simple lemma that allows us to state that "every deterministic vector at a distance smaller than the observable diameter to a deterministic equivalent is also a deterministic equivalent".
\begin{lemma}[\cite{LOU19}, Lemma 2.6.]\label{lem:trou_noir_diametre_observable}
  Given a sequence of random vectors $Z_n \in E_n$ and two sequence of deterministic random vector $\tilde Z_n, \tilde Z'_n \in E_n$, if $\|\tilde Z_n - \tilde Z'_n \| \leq O(\sigma_n)$, then:
  \begin{align*}
    Z \in \tilde Z \pm \mathcal E_2(\sigma)&
    &\Longleftrightarrow&
    &Z \in \tilde Z' \pm \mathcal E_2(\sigma).
  \end{align*}
\end{lemma}

\section{Linear concentration through sums and integrals}
Independence is known to be a key elements to most of concentration inequalities.
However, linear concentration behaves particularly well for the concatenation of random vectors whose dependence can not be disentangled.

The next proposition sets that the observable diameter for the $\ell^\infty$ norm remains unchanged through concatenation. Given a product $E \equiv \prod_{1\leq i\leq m} E_i$, where $(E_1,\|\cdot\|_\infty), \ldots, (E_m,\|\cdot\|_\infty)$ are $m$ normed vector spaces we define the $\ell^\infty$ norm on $E$ with the following identity:
\begin{align}\label{eq:definition_l_infty}
  (z_1,\ldots ,z_m) \in E: \ \Vert (z_1,\ldots , z_m) \Vert_{\ell^\infty}  = \sup_{1\leq i \leq m} \| z_i\|_i.
\end{align}
\begin{proposition}\label{pro:concentration_concatenation_vecteurs_lineairement_concentres}
    Given two sequences $m \in \mathbb N^{\mathbb N}$ and $\sigma \in \mathbb R_+^{\mathbb N}$, a constant $q$, $m$ sequences of normed vector spaces $(E_i, \|\cdot\|_i)_{1\leq i \leq m}$, $m$ sequences of deterministic vectors $\tilde Z_1\in E_1,\ldots, \tilde Z_m\in E_m$, and $m$ sequences of random vectors $Z_1 \in E_1,\ldots,Z_m \in E_m$ (possibly dependent) satisfying, for any $i\in\{1,\ldots ,m\}$, $Z_i \in \tilde Z_i \pm \mathcal E_2(\sigma)$, we have the concentration~:
  \begin{align*}
    (Z_1,\ldots,Z_m) \in (\tilde Z_1,\ldots, \tilde Z_m) \pm \mathcal E_2(\sigma),
     \ \ \text{in } (E, \Vert \cdot \Vert_{\ell^\infty}).
  \end{align*}
\end{proposition}
In other word, the linear observable diameter of $(Z_1,\ldots, Z_m)$ can not be bigger than the observable diameter of $(Z, \ldots,Z)$, where $Z$ is chosen as the worse possible random vector satisfying the hypotheses of $Z_1,\ldots, Z_m$.
\begin{remark}\label{rem:concatenation_impossible_concentration_convexe_ou_lipschitz}
  Example 2.27. in \cite{LOU19} shows that this stability towards concatenation is not true for Lipschitz and convex concentration.
\end{remark}
  \begin{proof}
    Let us consider a linear function $u: E \rightarrow \mathbb{R}$, such that $$\left\Vert u\right\Vert_{\infty} \equiv \sup_{\left\Vert z\right\Vert_{\infty}\leq 1} \vert u(z)\vert \leq 1.$$
    Given $i\in [m]$, let us note $u_i : E_i \rightarrow \mathbb R$ the function defined as $u_i(z) = u((0,\ldots,0,z,0,\ldots,0))$ (where $z$ is in the $i^{\text{th}}$ entry). For any $z\in E$, one can write~:
    \begin{align*}
      u(z) = \sum _{i=1}^m n_i u'_i(z_i), 
    \end{align*}
    where $n_i \equiv \|u_i\| = \sup_{\|z \|_i\leq 1} u_i(z)$ and $u'_i = u_i/n_i$ ($\| u'_i\| =1$). We have the inequality~:
    $$\sum _{i=1}^m n_i = \sum_{i=1}^m n_i\sup_{\|z_i\|_i \leq 1} u_i'(z_i) = \sup_{\Vert z \Vert_\infty \leq 1} u(z) \leq 1.$$
    With this bound at hand, we plan to employ the characterization with the centered moments. Let us conclude thanks to Proposition~\ref{pro:concentration_variable_caracterisation_moment} and the convexity of $t \mapsto t^l$, for any $l \geq 1$:
  \begin{align*}
    \mathbb{E}\left[\left\vert u(Z)-u(\tilde Z)\right\vert^l\right]
    &\leq \mathbb{E}\left[ \left( \sum_{i=1}^{m} n_i \left\vert u_i'\left(Z_i\right)-u'_i\left(\tilde Z_i\right)\right\vert \right)^l \right] \\
    &\leq \left( \sum_{i=1}^{m} n_i \right)^l \mathbb{E}\left[ \sum_{i=1}^{m} \frac{n_i}{ \sum_{i=1}^{m} n_i} \left\vert u_i'\left(Z_i\right)-u'_i\left(\tilde Z_i\right)\right\vert^l \right] \\
    &\leq \sup_{l\in[m]} \mathbb{E}\left[ \left\vert u_i'\left(Z_i\right)-u'_i\left(\tilde Z_i\right)\right\vert^l \right]\ \ \leq \ C l^{\frac{l}{2}} \sigma^l.
  \end{align*}
\end{proof}

If we want to consider the concatenation of vectors with different observable diameter, it is more convenient to look at the concentration in a space $(\prod_{i=1}^m E_i, \ell^r)$, for any given $r>0$, where, for any $(z_1,\ldots, z_m) \in \prod_{i=1}^m E_i$:
\begin{align*}
  \left\Vert (z_1,\ldots, z_m)\right\Vert_{\ell^r} = \left(\sum_{i=1}^m \|z_i\|_i^r\right)^{1/r}.
\end{align*}
\begin{corollaire}\label{cor:concentration_concatenation_vecteur_lineaireent_concentre}
   Given two constants $q,r>0$, $m \in \mathbb N^{\mathbb N}$, $\sigma_1,\ldots,\sigma_m \in (\mathbb R_+^{\mathbb N}) ^m$, $m$ sequences of $(E_i, \|\cdot\|_i)_{1\leq i \leq m}$, $m$ sequences of deterministic vectors $\tilde Z_1\in E_1,\ldots, \tilde Z_m\in E_m$, and $m$ sequences of random vectors $Z_1 \in E_1,\ldots,Z_m \in E_m$ (possibly dependent) satisfying, for any $i\in\{1,\ldots ,m\}$, $Z_i \in \tilde Z_i \pm \mathcal E_2(\sigma_i)$, we have the concentration~:
  \begin{align*}
    (Z_1,\ldots,Z_m) \in (\tilde Z_1,\ldots, \tilde Z_m) \pm \mathcal E_2(\|\sigma\|_r),
     \ \ \text{in } (E, \Vert \cdot \Vert_{\ell^r}),
  \end{align*}
\end{corollaire}
\begin{remark}\label{rem:concentration_concatenation_variable_lineairement_concentre}
  When $E_1=\cdots=E_m = E$ in the setting of Corollary~\ref{cor:concentration_concatenation_vecteur_lineaireent_concentre}, then for any vector $a=(a_1,\ldots,a_m)\in \mathbb R_+^m$, we know  that:
  \begin{align*}
    \sum_{i=1}^m a_i Z_i \in \sum_{i=1}^m a_i \tilde Z_i \pm \mathcal E_2(|a|^T\sigma),
  \end{align*}
  where $|a|=(|a_1|,\ldots,|a_m|)\in \mathbb R_+^m$
\end{remark}
\begin{proof}
  We already know from Proposition~\ref{pro:concentration_concatenation_vecteurs_lineairement_concentres} that:
  \begin{align*}
    \left(\frac{Z_1}{\sigma_1},\ldots,\frac{Z_m}{\sigma_m} \right) \in \left(\frac{\tilde Z_1}{\sigma_1},\ldots, \frac{\tilde Z_m}{\sigma_m}  \right) \pm \mathcal E_2,
     \ \ \text{in } (E, \Vert \cdot \Vert_{\ell^\infty}).
  \end{align*}
  Let us then consider the linear mapping:
  \begin{align*}
    \phi :
    \begin{aligned}[t]
      (E,\|\cdot \|_{\ell^\infty}) \ &
      &\longrightarrow&
      &(E,\|\cdot \|_{\ell^r})\hspace{0.6cm}\\
      (z_1,\ldots, z_m)&
      &\longmapsto&
      &(\sigma_1z_1,\ldots, \sigma_mz_m),
    \end{aligned}
  \end{align*}
  the Lipschitz character of $\phi$ is clearly $\|\sigma\|_r  = (\sum_{i=1}^m \sigma_i^r)^{1/r}$, and we can deduce the concentration of $Z = \phi(\sigma_1Z_1,\ldots, \sigma_m Z_m)$.  
\end{proof}
Corollary~\ref{cor:concentration_concatenation_vecteur_lineaireent_concentre} is very useful to set the concentration of infinite series of concentrated random variables. This is settled thanks to an elementary result of \cite{LOU19} that sets that the observable diameter of a limit of random vectors is equal to the limit of the observable vectors. Be careful that rigorously, there are two indexes, $n$ coming from Definition~\ref{def:linear_concentration} that only describes the concentration of sequences of random vectors, and $m$ particular to this lemma that will tend to infinity. For clarity, we do not mention the index $n$. 
\begin{lemma}[\cite{LOU19}, Proposition 1.12.]\label{lem:passage_a_la_limite_dans_la_concetration}
  Given a sequence of random vectors $(Z_{m})_{m \in \mathbb N} \in E^{\mathbb N}$, a sequence of positive reals $(\sigma_{m})_{m \in \mathbb N} \in \mathbb R^{\mathbb N}_+$ and a sequence of deterministic vectors $(\tilde Z_{m})_{m \in \mathbb N} \in E^{\mathbb N}$ such that:
  \begin{align*}
    Z_{m} \in \tilde Z_m \pm  \mathcal E_2(\sigma_{m}),
  \end{align*}
  if we assume that $(Z_{m})_{m \in \mathbb N}$ converges in 
  law\footnote{For any $n \in \mathbb N$, for any bounded continuous mapping $ f : \prod_{m\geq 0} E_p \to \mathbb R^{\mathbb N}$:
  \begin{align*}
     \sup_{n \in \mathbb N} \left\vert \mathbb E[f(Z_{n, m}) - \mathbb E[f(Z_{n, \infty})]\right\vert \underset{m\to \infty}{\longrightarrow} 0
   \end{align*} }
  when $m$ tends to infinity to a random vector $(Z_{\infty}) \in E$, that $\sigma_{m} \underset{n\to \infty}{\longrightarrow} \sigma_{\infty} $ and that $\tilde Z_{m} \underset{n\to \infty}{\longrightarrow} \tilde Z_{\infty}$, then:
  \begin{align*}
    Z_{\infty} \in \tilde Z_{\infty} \pm \mathcal E_2(\sigma_{\infty}).
  \end{align*}
  (The result also holds for Lipschitz and convex concentration)
\end{lemma}
\begin{corollaire}\label{cor:concentration_serie_vecteur_lineairement_concentres}
  \sloppypar{Given two constants $q,r>0$, $\sigma_1,\ldots,\sigma_m \ldots \in (\mathbb R_+^{\mathbb N}) ^\mathbb N$, a (sequences of) normed vector spaces $(E, \|\cdot\|)$, $\tilde Z_1 \ldots, \tilde Z_m,\ldots \in E^{\mathbb N}$ deterministic, and $ Z_1 \ldots,  Z_m,\ldots \in E^{\mathbb N}$ random (possibly dependent) satisfying, for any $n\in \mathbb N$, $Z_m \in \tilde Z_m \pm \mathcal E_2(\sigma_m)$. If we assume that $Z \equiv\sum_{n \in \mathbb N} Z_m$ is pointwise convergent\footnote{For any $w \in \Omega$, $\sum_{m \in \mathbb N} \|Z_m(w)\| \leq \infty$ and we define $Z(w) \equiv \sum_{m \in \mathbb N}Z_m(w)$}, that $\sum_{m \in \mathbb N} \tilde Z_m$ is well defined and that $\sum_{n\in \mathbb N} \sigma_i \leq \infty$, then we have the concentration~:
  \begin{align*}
    \sum_{m \in \mathbb N} Z_m \in \sum_{m \in \mathbb N} \tilde Z_m \pm \mathcal E_2 \left(\sum_{m \in \mathbb N} \sigma_m\right),
     \ \ \text{in } (E, \Vert \cdot \Vert),
  \end{align*}}
\end{corollaire}
\begin{proof}
  We already know from Corollary~\ref{cor:concentration_concatenation_vecteur_lineaireent_concentre} that for all $m \in \mathbb N$:
  \begin{align*}
    \sum_{m = 1}^M Z_m \in \sum_{m = 1}^M \tilde Z_m \pm \mathcal E_2 \left(\sum_{m \in \mathbb N} \sigma_m\right),
     \ \ \text{in } (E, \Vert \cdot \Vert).
  \end{align*}
  Thus in order to employ Lemma~\ref{lem:passage_a_la_limite_dans_la_concetration} let us note that for any bounded continuous mapping $f:E \to \mathbb R$, the dominated convergence theorem allows us to set that:
  \begin{align*}
    \mathbb E \left[ f \left(\sum_{m = 1}^M Z_m \right)\right] \underset{M \to \infty} \longrightarrow \mathbb E \left[ f \left(\sum_{m = 1}^\infty Z_m \right)\right],
  \end{align*}
  thus $(\sum_{m = 1}^M Z_m)_{N \in \mathbb N}$ converges in law to $\sum_{m = 1}^\infty Z_m$, which allows us to set the result of the corollary.

\end{proof}
The concentration of infinite series directly implies the concentration of resolvents and other related operators (like $(I_n - X/\sqrt p)^{-1} X^k$ for instance). 
\begin{corollaire}\label{cor:Concentration_linearire_solution_implicite_hypo_concentration_phi^k_pour tout_k}
  Given a (sequence of) vector space $(E, \| \cdot \|)$, let $\phi\in \mathcal A(E)$ be a (sequence of) random affine mapping such that there exists a constant $\varepsilon>0$  satisfying $\left\Vert\mathcal L(\phi)\right\Vert \leq 1-\varepsilon$ and a (sequences of) integers $\sigma >0$ satisfying for all (sequence of) integer $k$:
  \begin{align*}
    & \mathcal L(\phi)^k(\phi(0)) \in \mathcal E_2 \left(\sigma(1-\varepsilon)^k\right)  \ \ \text{in} \  \ (E, \| \cdot \|)
  \end{align*}
  Then the random equation
  \begin{align*}
    Y = \phi(Y)
  \end{align*}
  admits a unique solution $Y = ( Id_E - \mathcal L(\phi))^{-1} \phi(0)$ satisfying the linear concentration:
  \begin{align*}
    Y\in \mathcal E_2(\sigma)  .
  \end{align*}
\end{corollaire}
In practical examples, $\|\mathcal L(\phi)\|$ is rarely bounded by $1-\varepsilon$ for all drawings of $\phi$ and to obtain the concentration of $\mathcal L(\phi)^k$ with an observable diameter of order $\sigma(1- \varepsilon)^k$, one needs to place oneself on an event $\mathcal A_\phi$ satisfying $\mathcal A_\phi \subset \{\left\Vert\mathcal L(\phi)\right\Vert \leq 1-\varepsilon\}$. Then, thanks to a simple adaptation of Lemma~\ref{lem:concentration_conditionnee_convexe} below to the case of linear concentration, we have the concentration $(Y \ | \ \mathcal A_\phi)\in \mathcal E_2(\sigma)$. When $\mathbb E[\|\mathcal L(\phi)\|] \leq 1-2 \varepsilon$ for $\varepsilon\geq O(1)$ and $\phi$ is sufficiently concentrated, it is generally possible to chose an event $\mathcal A_\phi$ of overwhelming probability. 

As it will be seen in Subsection~\ref{sse:convex_concentration_resolvent}, this corollary finds its relevancy under convex concentration hypotheses, where the linear concentration seems to be the best concentration property to obtain on the resolvent $Q^z = (z I_p - \frac{1}{n}XX^T)^{-1}$.
\begin{proof}
  By contractivity of $\phi$, $Y$ is well defined and expresses:
  \begin{align*}
    Y = ( Id_E - \mathcal L(\phi))^{-1} \phi(0) = \sum_{k=0}^\infty \mathcal L(\phi)^k \phi(0).
  \end{align*}
  One can then conclude with Corollary~\ref{cor:concentration_serie_vecteur_lineairement_concentres} that $Y \in \mathcal E_2(\sigma/\varepsilon) = \mathcal E_2(\sigma)$.
\end{proof}

In order to satisfy the hypothesis of Corollary~\ref{cor:Concentration_linearire_solution_implicite_hypo_concentration_phi^k_pour tout_k}, but also for independent interest, we are now going to express the concentration of the product of convexly concentrated random matrices.

\section{Degeneracy of convex concentration through product}\label{sse:convex_concentration}

Given two convexly concentrated random vectors $X,Y \in E$ satisfying $X,Y \propto_c \mathcal E_2(\sigma)$, the convex concentration of the couple $(X,Y) \propto_c \mathcal E_2(\sigma)$ is ensured if:
\begin{enumerate}
  \item $X$ and $Y$ are independent
  \item $(X,Y) = u(Z)$ with $u$ affine and $Z\propto_c \mathcal E_2(\sigma)$.
\end{enumerate}
We can then in particular state the concentration of $X+Y$ as it is a linear transformation of $(X,Y)$. For the product it is not as simple as for the Lipschitz concentration, let us first consider the particular case of the entry-wise product in $E = \mathbb R^p$.
Since this result is not important for the rest of the paper, we left its proof in \ref{app:concentration_produit_vecteurs}.

\begin{theorem}\label{the:concentration_convexe_produit_odot_Rp}
  Given a (sequences of) integer $m \in \mathbb N^{\mathbb N}$ and a (sequence of) positive number $\sigma >0$ such that $m \leq O(p)$, a (sequence of) $m$ random vectors $X_1,\ldots,X_m \in \mathbb R^p$, if we suppose that 
  \begin{align*}
    X \equiv(X_1,\ldots,X_m) \propto_c \mathcal E_2(\sigma)
    &\text{ in } \left((\mathbb R^p)^m, \| \cdot \|_{\ell^\infty}\right),
   \end{align*}
   (with the notation $\|\cdot \|_{\ell^\infty}$ defined in \eqref{eq:definition_l_infty}) 
   and that there exists a (sequence of) positive numbers $\kappa >0$ such that $\forall i\in [m]: \|X_i \|_\infty \leq \kappa$, then:
   \begin{align*}
    X_1\odot \cdots \odot X_m \in \mathcal E_2 \left( (2e\kappa)^{m-1} \sigma\right)  \ \ \ \text{in } \ (\mathbb R^p, \|\cdot\|).
   \end{align*}
   And if $X_1=\cdots=X_m =X$, the constant $2e$ is no more needed and we get the concentration $X^{\odot m} \in \mathcal E_2 \left( \kappa^{m-1} \sigma\right)$.
\end{theorem}
\begin{remark}\label{rem:controle_de_la_queue_de_distribution}
  If we replace the strong assumption $\forall i\in [m]: \|X_i \|_\infty \leq \kappa$, with the bound $\sup_{1\leq i \leq m}\|\mathbb E[X_i] \|_\infty \leq O((\log p)^{1/q})$ we can still deduce a similar result to \cite[Example 4.]{LOU21HV}, stating the existence of a constant $\kappa \leq O(1)$ such that:
   \begin{align*}
    X_1\odot \cdots \odot X_m \in \mathcal E_2 \left( \left(\kappa \sigma\right)^{m} (\log(p))^{(m-1)/q}\right) + \mathcal E_{q/m} \left( \left(\kappa \sigma\right)^{m}\right)  \ \ \ \text{in } \ (\mathbb R^p, \|\cdot\|).
   \end{align*}
\end{remark}
The result of concentration of a product of matrices convexly concentrated was already proven in \cite{MEC11} but since their formulation is slightly different, we reprove in \ref{app:concentration_produit_vecteurs} the following result with the formulation required for the study of the resolvent.

\begin{theorem}[\cite{MEC11}, Theorem 1]\label{the:concentration_lineaire_produit_matrice_convexement_concentres}
  \sloppypar{Let us consider three sequences $m \in \mathbb N^{\mathbb N}$ and $\sigma,\kappa \in \mathbb R_ + ^{\mathbb N}$, and a sequence of $m$ random random matrices $X_1 \in \mathcal{M}_{n_0,n_1},\ldots, X_m \in \mathcal{M}_{n_{m-1},n_m}$, satisfying\footnote{The norm $\|\cdot \|_{F}$ is defined on $\mathcal{M}_{n_{0},n_1} \times \cdots\times \mathcal{M}_{n_{m-1},n_m}$ by the identity: $$\|(M_1,\ldots,M_p)\|_F = \sqrt{\|M_1\|_F^2 + \cdots + \|M_m\|_F^2}$$}:
  \begin{center}
    $(X_1,\ldots,X_m) \propto_c \mathcal E_2(\sigma)$ in $\left(\mathcal{M}_{n_{0},n_1} \times \cdots\times \mathcal{M}_{n_{m-1},n_m}, \| \cdot \|_F\right),$
    \end{center}
    In the particular case where $X_1=\cdots = X_n \equiv X$, it is sufficient\footnote{Be careful that $X \propto_c \mathcal E_2(\sigma)$ does not imply that $(X,\ldots, X) \propto_c \mathcal E_2(\sigma)$, it is only true when $(\mathcal{M}_{n})^m $ is endowed with the norm $\|\cdot \|_{F,\ell^\infty}$, satisfying for any $M= (M_1,\ldots,M_m) \in (\mathcal M_{n})^m$, $\|M\|_{F,\ell^\infty} = \sup_{1\leq i\leq m} \|M_i\|_F$} to assume that $X \propto_c \mathcal E_2(\sigma)$ in $(\mathcal{M}_{n}, \|\cdot \|_F)$.
   If there exist a sequence of positive values $\kappa>0$ such that $\forall i\in[m], \|X_i \| \leq \kappa $, then the product is concentrated for the nuclear norm:
   \begin{align*}
    X_1\cdots X_m \in \mathcal E_2 \left(\kappa^{m-1} \sigma \sqrt{n_0+\cdots + n_m}\right) &
    &\text{ in } \left(\mathcal M_{n_0,n_m}, \| \cdot \|_*\right),
   \end{align*}
   where, for any $M \in \mathcal M_{n_0,n_m}$, $\|M\|_* = \tr(\sqrt{MM^T})$ (it is the dual norm of the spectral norm). 
   }

\end{theorem}
\begin{remark}\label{rem:quand_X_non_borne}
  The hypothesis $\|X\|\leq \kappa$ might look quite strong, however in classical settings where $X \propto \mathcal E_2$ and $\|\mathbb E[X]\|\leq O(\sqrt n)$ it has been shown that there exist three constants $C,c, K >0$ such that $\mathbb P(\|X \| \geq K\sqrt n) \leq C e^{-cn}$. Placing ourselves on the event $\mathcal A = \{\|X \| \leq K\sqrt n\}$, we can then show from Lemma~\ref{lem:concentration_conditionnee_convexe} below that:
  \begin{align*}
    \left( (X/\sqrt n)^m \ | \ \mathcal A \right) \in \mathcal E_2 \left(K^{m-1}/\sqrt n\right)&
    &\text{and}&
    &\mathbb P(A^c) \leq C e^{-cn},
  \end{align*}
  (here $\sigma = 1/\sqrt n$ and $\kappa = K$). The same inferences hold for the concentration of $(XX^T/(n+p))^m$.
\end{remark}

We end this section on the concentration of the product of convexly concentrated random vectors with the Hanson-Wright Theorem that will find some use of the estimation of $\mathbb E[Q^z]$. This result was first proven in \cite{ADA14}, an alternative proof with our notations is provided in \cite[Proposition 8]{LOU21HV}\footnote{This paper only studies the Lipschitz concentration case, however, since quadratic forms are convex, the arguments stays the same with convex concentration hypotheses.}.
\begin{theorem}[\cite{ADA14}]\label{the:hanson_wright}
  Given two random matrices $X,Y \in \mathcal{M}_{p,n}$ such that $(X,Y) \propto_c \mathcal E_2$ and $\|\mathbb E[X]\|_F,\|\mathbb E[Y]\|_F\leq O(1)$, for any $ A \in \mathcal{M}_{p}$:
  \begin{align*}
    Y^TAX \in \mathcal E_2(\|A\|_F) + \mathcal E_1(\|A\|).
  \end{align*}
\end{theorem}

\section{Concentration of the resolvent of the sample covariance matrix of convexly concentrated data}\label{sse:convex_concentration_resolvent}
\subsection{Assumptions on $X$ and ``concentration zone'' of the resolvent}
Given $n$ data $x_1,\ldots, x_n \in \mathbb R^p$, to study the eigen behavior of the sample (non centered) covariance matrix $\frac{1}{n}XX^T$, where $X = (x_1,\ldots, x_n) \in \mathcal{M}_{p,n}$, one classically studies the resolvent $Q^z = (zI_p - \frac{1}{n} XX^T)^{-1} $ for the values of $z$ where it is defined. Let us note the $p$ eigen values of $\frac{1}{n}XX^T$: $\lambda_i = \sigma_i(\frac{1}{n}XX^T)$, for $i \in [p]$ ( then $\lambda_1\geq \cdots \geq \lambda_n$), then the spectral distribution of $\frac{1}{n}XX^T$:
\begin{align*}
  \mu = \frac{1}{p} \sum_{i=1}^p \delta_i
\end{align*}
has for Stieltjes transform $g : z \mapsto \frac{1}{p} \tr (Q^z)$. 

The present study was already lead in previous papers in the case of Lipschitz concentration of $X$ \cite{LOU21RHL}, or in the case of convex concentration of $X$ but with negative $z$ \cite{LOU19}. The goal of this section, is manly to present the consequences of Theorem~\ref{the:concentration_lineaire_produit_matrice_convexement_concentres} and adapt the recent results of \cite{LOU21RHL} on the case of convex concentration.
We adopt here classical hypotheses and assume a convex concentration for $X=(x_1,\ldots, x_n)$.
\begin{assumption}[Convergence scheme]\label{ass:p_plus_petit_que_n}
  $p = O(n)$.
\end{assumption}
\begin{assumption}[Independence]\label{ass:independence_x_i}
  $x_1,\ldots, x_n$ are independent.
\end{assumption}
\begin{assumption}[Concentration]\label{ass:concentration_covexe_X}
  $X \propto_c \mathcal E_2$.
\end{assumption}
\begin{assumption}[Bounding condition\footnote{As already done in \cite{LOU19} (but with real negative $z$), one can obtain the same conclusion assuming that there are a finite number of classes for the distribution of the columns $x_1,\ldots, x_n$ and that $\sup_{i\in [n]}\|\mathbb E[x_i]\| \leq O(\sqrt n)$}]\label{ass:borne_x_i}  $\sup_{i\in [n]}\|\mathbb E[x_i]\| \leq O(1)$.
\end{assumption}
When $n$ gets big, $\mu$ distributes along a finite number of bulks. To describe them, let us consider a positive parameter, $\varepsilon>0$, that could be chosen arbitrarily small (it will though be chosen independent with $n$ in most practical cases) and introduce as in \cite{LOU21RHL} the sets:
\begin{align*}
  \mathcal S = \left\{\lambda_i\right\}_{i \in [p]} &
  &\bar{\mathcal S} = \left\{ \mathbb E[\lambda_i]\right\}_{i \in [p]} &
  &\bar{\mathcal S}^\varepsilon = \left\{ x \in \mathbb R, \exists i \in [n], |x - \lambda_i | \leq \varepsilon \right\}
\end{align*}
One can show that $\nu \equiv \sup \bar{\mathcal S} = \mathbb E[\lambda_1] \leq O(1)$ and introducing the event:
\begin{align*}
  \mathcal A_\varepsilon \equiv \left\{ \forall i \in[p], \sigma_i \left( \frac{1}{n}XX^T \right) \in \bar{\mathcal S}^{\varepsilon/2} \right\},
\end{align*}
the concentration of $\sigma(X) / \sqrt n \in \mathbb E[\sigma(X)] \pm \mathcal E_2(1/\sqrt n)$, allows us to set:\footnote{In \cite{LOU21RHL}, the proof is conducted for Lipschitz concentration hypotheses on $X$. However, since only the linear concentration of $\sigma(X)$ is needed, the justification are the same in a context of convex concentration thanks to Theorem~\ref{the:concentrtaion_transversale_matrice_vecteurs}.}
\begin{lemma}[\cite{LOU21RHL}, Lemma 3.]\label{lem:P_A_overwhelming}
  There exist two constants $C,c>0$ such that $\mathbb P \left(\mathcal A^c\right) \leq C e^{-cn\varepsilon^2}$.
\end{lemma}
The following lemma allows us to conduct the concentration study on the highly probable event $\mathcal A_\varepsilon$ (when $\varepsilon \geq O(1)$).

\begin{lemma}\label{lem:concentration_conditionnee_convexe}
  Given a (sequence of) positive numbers $\sigma>0$, a (sequence of) random vector $Z\in E$ satisfying $Z \propto \mathcal E_2(\sigma)$, and a (sequence of) convex subsets $ A \subset E$, if there exists a constant $K>0$ such that $\mathbb P(Z \in  A)\geq K$ then:\footnote{There exist two constants $C,s>0$ such that for any (sequence of) $1$-Lipschitz and quasi-convex mappings $f:A\to \mathbb R$:
  \begin{align*}
    \forall t>0: \ \mathbb P \left(\left\vert f(Z) - \mathbb E[f(Z) \ | \ Z \in A]\right\vert \geq t \ | \ Z \in A \right) \leq C e^{-(t/c\sigma)^2},
  \end{align*}
  and similar concentration occur around any median of $f(Z)$ or any independent copy of $Z$ (under $\{Z \in A\}$).}
  \begin{align*}
     (Z | Z \in A) \propto_c \mathcal E_2(\sigma).
   \end{align*}
\end{lemma}
\begin{proof}
  The proof is the same as the one provided in \cite[Lemma 2.]{LOU21RHL} except that this time, one needs the additional argument that since $S = \{f \leq m_f\}$ (for $m_f$, a median of $f$) is convex, the mappings $z \mapsto d(z, S)$ and $z \mapsto -d(z, S)$ are both quasi-convex thanks to the triangular inequality.
\end{proof}

We can deduce from Lemma~\ref{lem:concentration_conditionnee_convexe} that for all $\varepsilon \geq O(1)$, $(X \ |\ \mathcal A_\varepsilon) \propto_c \mathcal E_2$, and the random matrix $(X \ |\ \mathcal A_\varepsilon)$ is far easier to control because $\|(X \ |\ \mathcal A_\varepsilon)\| \leq \nu + \frac{\varepsilon}{2}$ (we recall that $\nu \equiv \mathbb E[\lambda_1]$).

\subsection{Concentration of the resolvent}
Placing ourselves under the event $\mathcal A_\varepsilon$, let us first show that the resolvent $Q^z \equiv (zI_p - \frac{1}{n}XX^T)^{-1}$ is concentrated if $z$ has a big enough modulus. Be careful that the following concentration is expressed for the nuclear norm (for any deterministic matrix $A \in \mathcal{M}_{p}$ such that $\|A\|\leq O(1)$, $\tr(AQ^z) \in \mathcal E_2$). All the following results are provided under Assumptions~\ref{ass:p_plus_petit_que_n}-\ref{ass:borne_x_i}.
The next proposition is just provided as a first direct application of Theorem~\ref{the:concentration_lineaire_produit_matrice_convexement_concentres} and Corollary~\ref{cor:concentration_serie_vecteur_lineairement_concentres}, a stronger result is provided in Proposition~\ref{pro:Concentration_lineaire_Q_proche_spectre}.
\begin{proposition}\label{pro:concentration_resolvent_concentration_convexe_z_hors_de_boule}
  Given two parameters $\varepsilon>0$ and $z \in \mathbb C$ such that $|z| \geq \nu + \varepsilon$:
  \begin{align*}
    (Q^z \ | \ \mathcal A_\varepsilon) \in \mathcal E_2 \left( \frac{4}{\varepsilon} (\nu + \varepsilon )\right)&
    &\text{in} \ \ (\mathcal M_{p}, \| \cdot \|_*).
  \end{align*}
\end{proposition}
\begin{proof}
  We know from Lemma~\ref{lem:concentration_conditionnee_convexe} that $(X \ | \ \mathcal A_\varepsilon) \propto_c \mathcal E_2$ and from Theorem~\ref{the:concentration_lineaire_produit_matrice_convexement_concentres} that (here $\kappa = \nu + \frac{\varepsilon}{2} \leq O(1)$, $\sigma = 1/\sqrt n$ and $p= O(n)$):
  \begin{align*}
     \text{Under $\mathcal A_\varepsilon$:}&
     &\left(\frac{1}{n}XX^T\right)^m \in \mathcal E_2 \left( \left( \nu + \frac{\varepsilon}{2} \right)^{m} \sqrt m \right)&
    &\text{ in } \left(\mathcal M_{p}, \| \cdot \|_*\right).
   \end{align*}
   Let us then note that $\left(\nu + \frac{\varepsilon}{2} \right)^{m} \sqrt m  = O\left(\left(\nu + \frac{3\varepsilon}{4} \right)^{m}\right)$ and for $z \in \mathbb C$ satisfying our hypotheses: $(\nu + \frac{3\varepsilon}{4})/|z| \leq 1 - \frac{\varepsilon}{4(\nu + \varepsilon)}$.
   We can then deduce from Corollary~\ref{cor:concentration_serie_vecteur_lineairement_concentres} that under $\mathcal A_\varepsilon$:
   \begin{align*}
     Q^z = \frac{1}{z} \left(I_p - \frac{1}{zn}XX^T\right)^{-1} = \frac{1}{z}\sum_{i=1}^\infty \left(\frac{1}{zn}XX^T\right)^i \in \mathcal E_2 \left( \frac{4}{\varepsilon} (\nu + \varepsilon )\right).
   \end{align*}

\end{proof}
Let us now try to study the concentration of $Q^z$ when $z$ gets close to the spectrum, for that we now require $\varepsilon>0$ to be a constant ($\varepsilon \geq O(1)$).
\begin{proposition}\label{pro:Concentration_lineaire_Q_proche_spectre}
  Given $\varepsilon \geq O(1)$, for all $z \in \mathbb C \setminus \bar{\mathcal S}^ \varepsilon$:
  \begin{align*}
    (Q^z \ | \ \mathcal A_\varepsilon) \in \mathcal E_2&
    &\text{in } \ \ (\mathcal{M}_{p}, \|\cdot \|_*),
  \end{align*}
  and we recall that there exist two constants $C,c>0$ such that $ \mathbb P \left( \mathcal A_\varepsilon^c \right)  \leq C e^{-cn}$.
\end{proposition}
\begin{proof}
  Proposition~\ref{pro:concentration_resolvent_concentration_convexe_z_hors_de_boule} already set the result for $|z| \geq \nu + \varepsilon\equiv \rho $, therefore, let us now suppose that $|z| \leq \rho$. 

    With the notation $|Q^z|^2 \equiv \left( \Im(z)^2 + \left( \Re(z) - \frac{1}{n}XX^T \right)^2 \right)^{-1}$, let us decompose:
    \begin{align}\label{eq:decomposition_Q_Im_Re}
      Q^z = \left( \Re(z) - \frac{1}{n}XX^T \right) |Q^z|^2 - \Im(z) |Q^z|^2.
    \end{align}
    We can then deduce the linear concentration of $|Q^z|^2$ with the same justifications as previously thanks to the Taylor decomposition:
    \begin{align*}
      |Q^z|^2 
      = \frac{1}{\rho^2} \sum_{m=0}^\infty \left( 1 - \frac{\Im(z)^2}{\rho^2} - \frac{\left( \Re(z) - \frac{1}{n}XX^T \right)^2}{\rho^2} \right)^m.
    \end{align*}
    Indeed, $\|\Re(z)I_p - \frac{1}{n}XX^T\| \leq d(\Re(z), S)$ and $d(z, S)^2 = \Im(z)^2 + d(\Re(z), S)^2 \leq \rho$ thus:
    \begin{align*}
      \left\Vert 1 - \frac{\Im(z)^2}{\rho^2}  - \frac{1}{\rho^2}\left(\Re(z)I_p - \frac{1}{n}XX^T\right) ^2\right\Vert 
      & \leq 1 -  \frac{d(z, S)^2}{\rho^2}   
      \leq 1 - \frac{\varepsilon^2}{\rho^2} <1 .
    \end{align*}
  We therefore deduce from \eqref{eq:decomposition_Q_Im_Re} that:
  \begin{align*}
    (Q^z \ | \ \mathcal A_\varepsilon) 
    \in \mathcal E_2 \left(\frac{ 2}{\varepsilon^2} \left( |\Im(z)| + |\Re(z)| + \nu + \frac{\varepsilon}{2} \right)  \right) 
    = \mathcal E_2.
  \end{align*}
\end{proof}
For the sake of completeness, we left in the appendix an alternative laborious proof (but somehow more direct) already presented in \cite{LOU19}. 

\subsection{Computable deterministic equivalent}\label{sse:equivalent_deterministe_hypothese_convexe}

We are going to look for a deterministic equivalent of $Q$. We mainly follow the lines of \cite{LOU21RHL}, we thus allow ourselves to present the justifications rather succinctly. Although Proposition~\ref{pro:Concentration_lineaire_Q_proche_spectre} gives us a concentration of $Q^z$ in nuclear norm, we will provide a deterministic equivalent for the Frobenius norm with a better observable diameter. 
For any $z \in \mathbb C \setminus S^\varepsilon$, let us introduce $\bar \Lambda^z = (\tr(\Sigma_i\mathbb E[Q^z]))_{i \in [n]}$ and recall that for any $\delta \in \mathbb C^n$, we note $\tilde Q_\delta^z = (zI_p - \frac{1}{n} \sum_{i=1}^n \frac{\Sigma_i}{1-\delta_i})^z $. We have the following first approximation to $\mathbb E[Q^z]$:
\begin{proposition}\label{pro:first_deterministic_equivalent}
  For any $z \in \mathbb C \setminus \bar{\mathcal S}^\varepsilon$:
  \begin{align*}
    \left\Vert \tilde Q^z_{\bar \Lambda^z} \right\Vert \leq O(1)&
    &\text{and}&
    &\left\Vert \mathbb E[Q^z] - \tilde Q^z_{\bar \Lambda^z} \right\Vert_F \leq O \left( \frac{1}{\sqrt n} \right).
  \end{align*}
\end{proposition}
To prove this proposition, we will play on the dependence of $Q^z$ towards $x_i$ with the notation $X_{-i} \equiv (x_1, \ldots, x_{i-1}, 0, x_{i+1}, \ldots, x_n) \in \mathcal{M}_{p,n}$ and:
\begin{align*}
  Q_{-i}^z \equiv \left( zI_p - \frac{1}{n}X_{-i}X_{-i}^T \right)^{-1}.
\end{align*}
To link $Q^z$ to $Q^z_{-i}$ we will extensively use a direct application of the Schur identity:
\begin{align}\label{eq:Schur_identity}
   Q^zx_i = \frac{Q^z_{-i}x_i}{1 - \frac{1}{n}x_i^T Q^z_{-i}x_i}.
 \end{align} 
\begin{proof}
  All the estimations hold under $\mathcal A_\varepsilon$, therefore the expectation should also be taken under $\mathcal A_\varepsilon$ to be fully rigorous. Note that if $Q_{-i}$ and $x_i$ are independent on the whole universe, they are no more independent under $\mathcal A_\varepsilon$. However, since the probability of $\mathcal A_\varepsilon$ is overwhelming, the correction terms are negligible, we thus allow ourselves to abusively expel from this proof the independence and approximation issues related to $\mathcal A_\varepsilon$, a rigorous justification is provided in \cite{LOU21RHL}.

  Let us bound for any deterministic matrix $A \in \mathcal{M}_{p}$ such that $\|A\|_F\leq 1$:
  \begin{align*}
   \left\vert  \tr \left( A  \left( \mathbb E[Q^z] - \tilde Q^z_{\bar \Lambda^z} \right) \right) \right\vert
    &\leq \frac{1}{n} \sum_{i=1}^n \left\vert \mathbb E \left[ \tr \left( A \left( Q^z \left( \frac{\Sigma_i}{1+\bar \Lambda^z_i} - x_ix_i^T \right)\tilde Q^z_{\bar \Lambda^z} \right) \right) \right] \right\vert.
  \end{align*}
  We can then develop with \eqref{eq:Schur_identity}:
  \begin{align*}
    &\left\vert  \tr \left( A \left( \mathbb E[Q^z] - \tilde Q^z_{\bar \Lambda^z} \right) \right) \right\vert\\
    &\hspace{1cm}\leq \frac{1}{n} \sum_{i=1}^n \left\vert   \frac{ \tr \left( A \left( \mathbb E \left[ Q^z - Q_{-i}^z  \right]\Sigma_i\tilde Q^z_{\bar \Lambda^z} \right) \right)}{1 - \bar \Lambda^z_i}  \right\vert\\
    &\hspace{1cm}\hspace{1cm} + \frac{1}{n} \sum_{i=1}^n \left\vert \mathbb E \left[ \tr \left( A \left( Q_{-i}^z \left( \frac{\Sigma_i}{1 - \bar \Lambda^z_i} - \frac{x_ix_i^T}{1 + \frac{1}{n} x_i^T Q_{-i}^z x_i} \right)\tilde Q^z_{\bar \Lambda^z} \right) \right) \right] \right\vert\\
    &\hspace{1cm}\leq \frac{1}{n} \sum_{i=1}^n \left\vert \mathbb E \left[ \frac{ x_i^T\tilde Q^z_{\bar \Lambda^z}A Q^zx_i}{1 - \bar \Lambda^z_i } \left( \frac{1}{n }x_i^t Q_{-i}^z x_i - \bar \Lambda^z_i \right)\right] \right\vert + O \left(  \frac{\left\Vert \tilde Q^z_{\bar \Lambda^z} \right\Vert}{\sqrt n}\right),
  \end{align*}
  thanks to Lemma~\ref{lem:Q_m_Q_m_i} and the independence between $Q_{-i}^z$ and $x_i$.
  We can then bound thanks to H\"older inequality and Lemma~\ref{lem:concentration xQx} below:
  \begin{align*}
    &\left\vert \mathbb E \left[ x_i^T\tilde Q^z_{\bar \Lambda^z}A Q^zx_i \left( \frac{1}{n }x_i^t Q_{-i}^z x_i - \bar \Lambda^z_i \right) \right] \right\vert\\
    &\hspace{0.5cm}=\left\vert \mathbb E \left[ \left( x_i^T\tilde Q^z_{\bar \Lambda^z}A Q^zx_i - \mathbb E \left[ x_i^T\tilde Q^z_{\bar \Lambda^z}A Q^zx_i \right] \right) \left( \frac{1}{n }x_i^t Q_{-i}^z x_i - \mathbb E \left[ \frac{1}{n }x_i^t Q_{-i}^z x_i \right] \right) \right] \right\vert\\
  &\hspace{0.5cm}\leq \sqrt{\mathbb E \left[ \left(\frac{x_i^T AQ^z_{-i}x_i}{1 - \frac{1}{n} x_i^TQ_{-i}x_i} - \mathbb E \left[ \frac{x_i^T AQ^z_{-i}x_i}{1 - \frac{1}{n} x_i^TQ_{-i}x_i} \right] \right)^2 \right]}O \left( \frac{1}{\sqrt n}  \right)\\
  &\hspace{0.5cm}\leq O \left( \frac{1}{\sqrt n}  \right)\left(\mathbb E \left[ \left(\frac{x_i^T \tilde Q^z_{\bar \Lambda^z}AQ^z_{-i}x_i - \mathbb E[x_i^T \tilde Q^z_{\bar \Lambda^z}AQ^z_{-i}x_i]}{1 - \frac{1}{n} x_i^TQ_{-i}x_i} \right)^2 \right] \right.\\
  &\hspace{1.5cm}\left.+ \mathbb E \left[ \left(\tr\left(\Sigma_i \tilde Q^z_{\bar \Lambda^z}AQ^z_{-i}\right) \left( \frac{1}{1 - \frac{1}{n} x_i^TQ_{-i}x_i} - \mathbb E \left[ \frac{1}{1 - \frac{1}{n} x_i^TQ_{-i}x_i} \right] \right) \right)^2 \right]\right)^{1/2}\\
  &\hspace{0.5cm}\leq O \left( \frac{\left\Vert \tilde Q^z_{\bar \Lambda^z} \right\Vert}{\sqrt n}  \right).
  \end{align*}
  indeed since we know that $|\frac{1}{1 - \frac{1}{n} x_i^TQ_{-i}x_i} | \leq O(1)$ from Lemma~\ref{lem:Q_borne}, $\frac{1}{1 - \frac{1}{n} x_i^TQ_{-i}x_i}$ is a $O(1)$-Lipschitz transformation of $\frac{1}{n} x_i^TQ_{-i}x_i$, therefore, it follows the same concentration inequality (with a variance of order $O(1/ n)$).
  Since this inequality is true for any $A \in \mathcal M_p$, we can bound:
  \begin{align*}
    \left\Vert \tilde Q^z_{\bar \Lambda^z} \right\Vert \leq \left\Vert \tilde Q^z_{\bar \Lambda^z} - \mathbb E[Q^z] \right\Vert_F + \left\Vert \mathbb E[Q^z] \right\Vert \leq O \left( \frac{\left\Vert \tilde Q^z_{\bar \Lambda^z} \right\Vert}{\sqrt n} \right) + O(1),
  \end{align*}
  which directly implies that $\left\Vert \tilde Q^z_{\bar \Lambda^z} \right\Vert \leq O(1)$ and $\left\Vert \mathbb E[Q^z] -  \tilde Q^z_{\bar \Lambda^z} \right\Vert_F \leq O(1/\sqrt n)$.
\end{proof}
\begin{lemma}[\cite{LOU21RHL}, Lemmas 4., 8. ]\label{lem:Q_borne}
  $\forall z \in \bar{\mathcal S}^\varepsilon$, under $\mathcal A^\varepsilon$: 
  \begin{align*}
    \|Q^z \| \leq \frac{2}{\varepsilon}&
    &\text{and}&
    &\sup_{i\in [n]}|\frac{1}{1 - \frac{1}{n} x_i^TQ_{-i}x_i} | \leq O(1)
  \end{align*}
\end{lemma}
\begin{lemma}\label{lem:uQx}
  For any $z \in \mathbb C \setminus \bar{\mathcal S}^\varepsilon$, any $i \in [n]$ and any $u \in \mathbb R^p$ such that $\|u\|\leq 1$:
  \begin{align*}
    (u^T Q_{-i}^zx_i \ | \ \mathcal A_\varepsilon),
    (u^T Q^zx_i \ | \ \mathcal A_\varepsilon)
     \in O(1) \pm \mathcal E_2.
  \end{align*}
\end{lemma}
\begin{proof}
    We do not care about the independence issues brought by $\mathcal A_\varepsilon$. Let us simply bound for any $t>0$ and under $\mathcal A_\varepsilon$:
    \begin{align*}
      &\mathbb P \left( \left\vert u^T Q_{-i}^zx_i - \mathbb E \left[  u^T Q_{-i}^zx_i\right] \right\vert \geq t \right)\\
      &\hspace{1cm}\leq \mathbb P \left( \left\vert u^T Q_{-i}^z(x_i -\mu_i) \right\vert \geq \frac{t}{2} \right)+
      \mathbb P \left( \left\vert u^T  \left( Q_{-i}^z-\mathbb E \left[ Q_{-i}^z\right] \right)\mu_i \right\vert \geq \frac{t}{2} \right)\\
      &\hspace{1cm}\leq \mathbb E \left[ C e^{-cnt^2/\|Q_{-i}\|^2} \right] + C e^{-cnt^2} \ \
      \leq \ 2C e^{-c'nt^2},
    \end{align*}
    for some constants $C,c,c'>0$. Besides, we can bound:
    \begin{align*}
      \left\vert \mathbb E \left[  u^T Q_{-i}^zx_i\right] \right\vert = \left\vert u^T \mathbb E[Q_{-i}^z]\mu_i \right\vert \leq O(1),
    \end{align*}
    thanks to Lemma~\ref{lem:Q_borne} and Assumption~\ref{ass:borne_x_i}.

    The concentration of $u^T Q^zx_i$ is a consequence of the concentration $QX \in \mathcal E_2$ that can be shown thanks to Corollary~\ref{cor:concentration_serie_vecteur_lineairement_concentres} as in the proof of Proposition~\ref{pro:Concentration_lineaire_Q_proche_spectre}.
    We are then left to bounding $\mathbb E[u^T Q^zx_i]$. For this purpose, let us write:
    \begin{align*}
      \left\vert \mathbb E[u^T Q^zx_i]  \right\vert
      &= \left\vert \mathbb E[u^T Q_{-i}^zx_i] - \mathbb E \left[   (u^T Q_{-i}^zx_i) \left( \frac{1}{n} x_i^TQ^zx_i \right) \right] \right\vert \\
      &\leq O(1) + O \left( \sqrt{\mathbb E \left[   (u^T Q_{-i}^zx_i)^2 \right]\mathbb E \left[ \left( \frac{1}{n} x_i^TQ^zx_i \right)^2 \right]} \right) \leq O(1),
    \end{align*}
    thanks to Cauchy-Schwarz inequality Lemma~\ref{lem:Q_borne}, and the bound on $x_i$, valid under $\mathcal A_\varepsilon$.
\end{proof}
\begin{lemma}\label{lem:Q_m_Q_m_i}
  Under $\mathcal A_\varepsilon$, for any $z \in \mathbb C \setminus \bar{\mathcal S}^\varepsilon$ and any $i \in [n]$:
  \begin{align*}
    \|\mathbb E[Q^z - Q^z_{-i}]\| \leq O \left( \frac{1}{n} \right).
  \end{align*}
\end{lemma}
\begin{proof}
  For any $u \in \mathbb R^p$, we can bound thanks to Lemma~\ref{lem:uQx}:
  \begin{align*}
    \left\vert u^T \mathbb E[Q^z - Q^z_{-i}]u \right\vert
    &\leq  \frac{1}{n}\left\vert  \mathbb E \left[ u^TQ^zx_ix_i^TQ^z_{-i}u \right] \right\vert \\
    &\leq  \frac{1}{n}\sqrt{\mathbb E \left[ (u^TQ^zx_i)^2\right]\mathbb E \left[ (x_i^TQ^z_{-i}u)^2 \right] } \ \ \leq O \left( \frac{1}{n} \right).
  \end{align*}
\end{proof}

\begin{lemma}\label{lem:concentration xQx}
  For any $z \in  \mathbb C \setminus \bar{\mathcal S}^\varepsilon $deterministic matrix $A \in \mathcal{M}_{p}$:
  \begin{align*}
     (x_i^T AQ^z_{-i}x_i \ | \ \mathcal A_\varepsilon) \in \tr(\Sigma_i A\mathbb E[Q^z]) \pm\mathcal E_2 \left( \|A\|_F \right) + \mathcal E_1(\|A\|).
  \end{align*} 
\end{lemma}

\begin{proof}
Once again, without referring to $\mathcal A_\varepsilon$, we assume that $\|X\|\leq O(1)$ and $\|Q^z\| \leq O(1)$. Given $i \in [n]$, since we know from Lemma~\ref{lem:Q_m_Q_m_i} that $\|\mathbb E[Q^z - Q_{-i}^z\| \leq O(1/\sqrt n)$, we want to bound:
\begin{align*}
  \left\vert x_i^T AQ^z_{-i}x_i - \tr \left( \Sigma_i A \mathbb E \left[ Q_{-i}\right] \right)\right\vert
  &\leq \left\vert x_i^T AQ^z_{-i}x_i -  \tr (\Sigma_iA Q^z_{-i})\right\vert + \left\vert  \tr \left(  \Sigma_iA (Q^z_{-i} - \mathbb E[Q^z_{-i}]) \right)\right\vert.
\end{align*}
Now we know that, for $X_{-i}$ fixed, we can bound thanks to Theorem~\ref{the:hanson_wright}:
\begin{align*}
  \mathbb P \left( \left\vert x_i^T AQ^z_{-i}x_i -  \tr (\Sigma_i^T AQ^z_{-i})\right\vert \geq t  \right)
  &\leq \mathbb E \left[ Ce^{-c(t/\|Q_{-i}^z\| \|A\|_F)^2} + Ce^{-ct/\|Q_{-i}^z\| \|A\|} \right] \\
  &\leq  Ce^{-c't^2/\|A\|_F^2} + Ce^{-c't/\|A\|},
\end{align*}
for some constants $C,c,c'>0$, thanks to Lemma~\ref{lem:Q_borne}.

Besides, we know from Proposition~\ref{pro:Concentration_lineaire_Q_proche_spectre} and Lemma~\ref{lem:trou_noir_diametre_observable} that $Q_{-i}^z \in \mathbb E[Q^z] \pm \mathcal E_2(1/\sqrt n)$ in $(\mathcal{M}_{p}, \|\cdot\|_*)$, which allows us to bound:
\begin{align*}
  \mathbb P \left( \left\vert \tr (\Sigma_i A Q^z_{-i}) - \tr (\Sigma_i A \mathbb E[Q^z])\right\vert \geq t  \right)
  \leq  Ce^{-ct^2/\|A\|^2},
\end{align*}
for some constants $C,c>0$, since $\|\Sigma_i\|\leq O(1)$. Putting the two concentration inequalities together, we obtain the result of the lemma.

\end{proof}

 Theorem~\ref{the:concentration_resolvent_main} is then a consequence of the following proposition proven in \cite{LOU21RHL} (once Proposition~\ref{pro:first_deterministic_equivalent} is proven, the convex concentration particularities do not intervene anymore). Recall that $\tilde \Lambda^z \in \mathbb C^n$ is defined as the unique solution to the equation:
 \begin{align*}
   \forall i\in [n]: \tilde \Lambda^z_i = \frac{1}{n} \tr \left(\Sigma_i \tilde Q^z_{\tilde \Lambda^z} \right),
 \end{align*}
 where $\tilde Q^z_{\tilde \Lambda^z} \equiv \left( zI_p - \frac{1}{n}\sum_{i=1}^n \frac{\Sigma_i}{1 - \tilde \Lambda_i^z} \right)$.

\begin{proposition}\label{pro:deuxième équivalent déterministe}
  For all $z \in \mathbb C \setminus \bar{\mathcal S}_\varepsilon$:
  \begin{align*}
    \left\Vert \mathbb E[Q] - \tilde Q^z_{\tilde \Lambda^z} \right\Vert_F \leq O \left( \frac{1}{\sqrt n} \right).
  \end{align*}
\end{proposition}

\appendix
\section{Proofs of the concentration of products of convexly concentrated random vectors and of convexly concentrated random matrices}\label{app:concentration_produit_vecteurs}
We will use several time the following elementary result:
\begin{lemma}\label{lem:sum_f_convexe_convexe}
  Given a convex mapping $f : \mathbb R \to \mathbb R$, and a vector $a \in \mathbb R_+^p$, the mapping $F : \mathbb R^p \ni (z_1,\ldots,z_p) \mapsto \sum_{i=1}^p a_if(z_i) \in \mathbb R$ is convex (so in particular quasi-convex).
\end{lemma}
To efficiently manage the concentration rate when multiplying a large number of random vectors, we will also need:
\begin{lemma}\label{lem:decomposition_poly_symetrique}
  Given $m$ commutative or non commutative variables $a_1,\ldots,a_m$ of a given algebra, we have the identity:
  \begin{align*}
     \sum_{\sigma \in \mathfrak S_m} a_{\sigma(1)}\cdots a_{\sigma(m)} = (-1)^m\sum_{I\subset [m]} (-1)^{|I|} \left(\sum_{i\in I} a_i\right)^m,
  \end{align*}
  where $|I|$ is the cardinality of $I$.
\end{lemma}
\begin{proof}
  The idea is to inverse the identity:
  \begin{align*}
    (a_1+\cdots+a_m)^m = \sum_{J\subset I} \sum_{ \{i_1,\ldots,i_m\} =  J}a_{i_1} \cdots a_{i_m},
  \end{align*}
  thanks to the Rota formula (see \cite{Rol06}) that sets for any mappings $f, g$ defined on the set subsets of $\mathbb N$ and having values in a commutative group (for the sum):
  \begin{align*}
    \forall I \subset \mathbb N, f(I) = \sum_{J\subset I} g(J)&
    &\Longleftrightarrow&
    & \forall I \subset \mathbb N, g(I) = \sum_{J\subset I}\mu_{\mathcal P(\mathbb N)}(J,I) f(J),
  \end{align*}
  where $\mu_{\mathcal P(\mathbb N)}(J,I)  =(-1)^{|I\setminus J|}$ is an analog of the Mo\"ebus function for the order relation induced by the inclusions in $\mathcal P(\mathbb N)$.
  In our case, for any $J \subset [m]$, if we set:
  \begin{align*}
    f(J) = \left(\sum_{i\in J} a_i \right)^m&
    &\text{and}&
    &g(J) = \sum_{ \{i_1,\ldots,i_m\} =  J}a_{i_1} \cdots a_{i_m},
  \end{align*}
  we see that for any $I \subset [m]$, $f(I) = \sum_{J\subset I} g(J)$, therefore taking the Rota formula in the case $I = [m]$, we obtain the result of the Lemma (in that case, $\mu_{\mathcal P(\mathbb N)}(J,I)= (-1)^{m-|J|}$ and $\sum_{ \{i_1,\ldots,i_m\} =  I}a_{i_1} \cdots a_{i_m} = \sum_{\sigma \in \mathfrak S_m} a_{\sigma(1)}\cdots a_{\sigma(m)}$).
\end{proof}

\begin{proof}[Proof of Theorem~\ref{the:concentration_convexe_produit_odot_Rp}]
  Let us first assume that all the $X_i$ are equal to a vector $Z \in \mathbb R^p$. 
  Considering $a = (a_1,\ldots, a_p) \in \mathbb R^p$, we want to show the concentration of $a^T Z^{\odot m} = \sum_{i=1}^p a_i z_i^m$ where $z_1,\ldots,z_p$ are the entries of $Z$. 
  
  The mapping $p_m : x \mapsto x^m $ is not quasi-convex when $m$ is odd, therefore, in that case we decompose it into the difference of two convex mappings $p_m(z) = p_m^+(z) - p_m^-(z)$ where:
  \begin{align}\label{eq:decomposition_convexe_puissance}
    p_m^+: z \mapsto \max(z^m,0)&
    &\text{and}&
    &p^-_m: z \mapsto -\min(z^m,0),
   \end{align} 
  (say that, if $m$ is even, then we set $p_m^+ = p_m$ and $p_m^- : z \mapsto 0$).
  For the same reasons, we decompose $\phi^+_a : z \mapsto a^T p_m^+(z)$ and $\phi^-_a : z \mapsto a^T p_m^-(z)$ into:  
  \begin{align*}
     \phi^+_a = \phi^+_{| a|} - \phi^+_{|a|- a}&
     &\text{and}&
     &\phi^-_a = \phi^-_{| a|} - \phi^-_{|a|- a}
   \end{align*}
   (for $|a| = (|a_i|)_{1\leq i \leq p}$), so that:
  \begin{align*}
    a^T Z^{\odot m} = \phi^+_{| a|}(Z) - \phi^+_{|a|- a}(Z) - \phi^-_{| a|}(Z) + \phi^-_{|a|- a}(Z)
  \end{align*}
  becomes a combination of quasi-convex functionals of $Z$. We now need to measure their Lipschitz parameter. Let us bound for any $z \in \mathbb R^p$:
  \begin{align*}
    \left\vert \phi^+_{| a|}(z)\right\vert = \sum_{i=1}^n |a_i| |z_i|^m \leq \|a\|\|z\| \|z\|_\infty^{m-1},
  \end{align*}
  and the same holds for $\phi^+_{|a|- a}$, $\phi^-_{| a|}$ and $\phi^-_{|a|- a}$. 
  Note then that $\phi^+_{|a|}$, $\phi^+_{|a|- a}$, $\phi^-_{| a|}$ and $\phi^-_{|a|- a}$ are all $\|a\| \kappa^{m-1}$-Lipschitz to conclude on the concentration of $X^{\odot m}$.

  Now, if we assume that the $X_1,\ldots, X_m$ are different, we employ Lemma~\ref{lem:decomposition_poly_symetrique} in this commutative case to write ($|\mathfrak S_m | = m!$):
  \begin{align}\label{eq:decomposition_x_1_odot_x_m}
     (X_1\odot \cdots \odot X_m) = \frac{(-1)^m}{m!} \sum_{I\subset [m]} (-1)^{|I|} \left(\sum_{i\in I} X_i\right)^{\odot m}.
   \end{align} 
  Therefore, the sum $(\mathbb R^p)^I \ni z_1,\ldots, z_{i_{|I|}}  \mapsto \sum_{i\in I} z_i \in \mathbb R^p$ being $m$-Lipschitz for the norm $\|\cdot \|_\infty$, we know that $\forall I \subset [m]$, $\sum_{i\in I} X_i \propto_c \mathcal E_2(m\sigma) $, and $\|\sum_{i\in I} X_i\|_\infty \leq \kappa m$, therefore, $(\sum_{i\in I} X_i)^{\odot m} \in \mathcal E_2(m^m\kappa^{m-1} \sigma )$.
  We can then exploit Proposition~\ref{pro:concentration_concatenation_vecteurs_lineairement_concentres} to obtain
  \begin{align*}
    \left(\left(\sum_{i\in I} X_i\right)^{\odot m}\right)_{I \subset [m]} \in \mathcal E_2(m^m\kappa^{m-1} \sigma ) &
    &\text{in } \ \left((\mathbb R^p)^{2^m}, \| \cdot \|_{\ell^\infty}\right),
  \end{align*}
  (note that $\#\{I \subset [m]\} = 2^m$)
  Thus summing the $2^m$ concentration inequalities, we can conclude from Equation~\eqref{eq:decomposition_x_1_odot_x_m}, and the Stirling formula $\frac{m^m}{m!} = \frac{e^m}{\sqrt{2\pi m}} +O(1)$ that:
  \begin{align*}
     (X_1\odot \cdots \odot X_m) \in \mathcal E_2 \left((2e\kappa)^{m-1} \sigma\right).
  \end{align*}

For the concentration of the matrix product, we introduce a new notion of concentration, namely the \textit{transversal convex concentration}. Let us give some definitions.
\begin{definition}\label{def:concentrtaion_transversale}
  
  Given a sequence of normed vector spaces $(E_n, \Vert \cdot \Vert_n)_{n\geq 0}$, a sequence of groups $(G_n)_{n\geq 0}$, each $G_n$ (for $n \in \mathbb N$) acting on $E_n$, a sequence of random vectors $(Z_n)_{n\geq 0} \in \prod_{n\geq 0} E_n$, a sequence of positive reals $(\sigma_n)_{n\geq 0} \in \mathbb R_+ ^{\mathbb N}$, we say that $Z =(Z_n)_{n\geq 0}$ 
  is convexly concentrated transversally to the action of $G$ with an observable diameter of order $\sigma$ and we note $Z \propto^T_G \mathcal E_2(\sigma)$ iff there exist two constants $C,c \leq O(1)$ such that $\forall n \in \mathbb N$ and for any $1$-Lipschitz, quasi-convex and $G$-invariant\footnote{For any $g \in G$ and $x \in E$, $f(x) = f(g \cdot x)$} function $f : E_n \to \mathbb R$,
  $\forall t>0:$\footnote{Once again, we point out that one could have replaced here $\mathbb E[f(Z_n)]$ by $f(Z_n')$ or $m_f$.} 
     \begin{center}
      $\mathbb P \left(\left\vert f(Z_n) - \mathbb E[f(Z_n)]\right\vert\geq t\right) \leq C e^{-c(t/\sigma_n)^2}$.
     \end{center}
\end{definition}
\begin{remark}\label{rem:concentration_convexe_implique_concentration_transversale}
  Given a normed vector space $(E,\|\cdot \|)$, a group $G$ acting on $E$ and a random vector $Z \in E$, we have the implication chain:
  \begin{align*}
    Z \propto \mathcal E_2(\sigma)&
    &\Longrightarrow&
    &Z \propto_c \mathcal E_2(\sigma)&
    &\Longrightarrow&
    &Z \propto_G^T \mathcal E_2(\sigma).
  \end{align*}
\end{remark}
 
Considering the actions:
\begin{itemize}
  \item $\mathfrak S_n$ on $\mathbb R^p$ where for $\sigma \in \mathfrak S_n$ and $x \in \mathbb R^p$, $\sigma \cdot x = (x_{\sigma(i)})_{1\leq i \leq p}$,
  \item $\mathcal O_{p,n} \equiv \mathcal O_{p} \times \mathcal O_{n}$ on $\mathcal M_{p,n}$ where for $(P,Q) \in \mathcal O_{p,n}$ and $M \in \mathcal M_{p,n}$, $(P,Q) \cdot M = PMQ$,
\end{itemize}
the convex concentration in $\mathcal M_{p,n}$ transversally to $\mathcal O_{p,n}$ can be expressed as a concentration on $\mathbb R^p$ transversally to $\mathfrak S_n$ thanks to the introduction the mapping $\sigma$ providing to any matrix the ordered sequence of its singular values~:
\begin{align*}
  \sigma \ : \
  \begin{aligned}[t]
    &\mathcal{M}_{p,n} &&\rightarrow && \hspace{1.4cm}\mathbb{R}^d_+ \\
    &M&&\mapsto&&(\sigma_1(M),\ldots, \sigma_d(M)).
  \end{aligned}&
  &\text{with } \ d = \min(p,n)
\end{align*}
(there exists $(P,Q) \in \mathcal O_{p,n}$ such that $M = P \Sigma(M) Q$, where $\Sigma \in \mathcal M_{p,n}$ has $\sigma_1(M)\geq \cdots \geq\sigma_d(M)$ on the diagonal).
\begin{theorem}[\cite{LED05}, Corollary 8.23. \cite{LOU19}, Theorem 2.44]\label{the:concentrtaion_transversale_matrice_vecteurs}
  Given a random matrix $Z \in \mathcal M_{p,n}$:
  \begin{align*}
    Z \propto^T_{\mathcal O_{p,n}} \mathcal E_2(\sigma)&
    &\Longleftrightarrow&
    &\sigma(Z) \propto^T_{\mathfrak S_d} \mathcal E_2(\sigma)
  \end{align*}
  (where the concentrations inequalities are implicitly expressed for euclidean norms: $\|\cdot\|_F$ on $\mathcal M_{p,n}$ and $\|\cdot\|$ on $\mathbb R^d$).
\end{theorem}

\begin{proof}[Proof of Theorem~\ref{the:concentration_lineaire_produit_matrice_convexement_concentres}]
  Let us start to study the case where $ X_1 = \cdots = X_m \equiv X \in \mathcal{M}_{n}$ and $X \propto \mathcal E_2$ in $(\mathcal{M}_{n}, \|\cdot \|_F)$. 
  We know from Theorem~\ref{the:concentrtaion_transversale_matrice_vecteurs} that:
  \begin{align*}
    \sigma(X) \propto_{\mathfrak S_p}^T \mathcal E_2,
  \end{align*}
  and therefore, as a $\sqrt{n}$-Lipschitz linear observation of $\sigma(X)^{\odot m} \in \mathcal E_2 \left( \kappa^{m-1} \sigma\right)$ (see Theorem~\ref{the:concentration_convexe_produit_odot_Rp}), $\tr(X^m)$ follows the concentration:
\begin{align*}
    \tr(X^m) = \sum_{i=1}^p \sigma_i(X)^m \in \mathcal E_2 \left( \sqrt n \kappa^{m-1} \sigma\right) .
\end{align*}

  Now, we consider the general setting where we are given $m$ matrices $X_1,\ldots, X_m$, a deterministic matrix $A \in \mathcal{M}_{n_n,n_0}$ satisfying $\|A\|\leq 1$, and we want to show the concentration of $tr(AX_1,\cdots, X_m)$. First note that we stay in the hypotheses of the theorem if we replace $X_1$ with $AX_1$, we are thus left to show the concentration of $\tr(X_1 \cdots X_m)$.
  We can not employ again Lemma~\ref{lem:decomposition_poly_symetrique} without a strong hypothesis of commutativity on the matrices $X_1,\ldots, X_n$. Indeed, one could not have gone further than a concentration on the whole term $\sum_{\sigma \in \mathfrak S_p}\tr(X_{\sigma(1)} \cdots X_{\sigma(m)})$.
However, we can still introduce the random matrix
    \begin{align*}
    Y = \left( \begin{array}{cccc}
    0&X_{m-1} &&  \\
    & \ddots&\ddots& \\
    & &\ddots&X_1 \\
    X_m& &&0  \end{array} \right)&
    &\text{then}&
    &Y^m = \left( \begin{array}{cccc}
    0&X^{m}_1 &&  \\
    & \ddots&\ddots& \\
    & &\ddots&X_{3}^{2}\\
    X_{2}^{1}& &&0  \end{array} \right),
    \end{align*}
    where for $i,j \in \{2, \ldots, m-1\}$, $X^j_i \equiv X_i X_{i+1} \cdots X_m X_{1}\cdots X_j$. Since $Y \in \mathcal{M}_{n_0+\cdots+n_m}$ satisfies $Y \propto \mathcal E_2(\sigma)$ and $\|Y\|\leq \kappa$, the first part of the proof provides the concentration $Y^m \in \mathcal E_2 \left(\kappa^{m-1}\sigma\sqrt{n_0+\cdots+n_m}\right)$ in $\left(\mathcal M_n, \| \cdot \|_*\right)$ which directly implies the concentration of $X_1^m= X_1\cdots X_m$.

\end{proof}

\end{proof}

\section{Alternative proof of Proposition~\ref{pro:Concentration_lineaire_Q_proche_spectre}}

  We are going to show the concentration of the real part and the imaginary part of $Q^z$, where:
  \begin{align*}
    &\Re(Q^z) = Q^z \left(\Re(z) I_p - \frac{1}{n}XX^T\right)\bar{Q}^z =(\Re(z)-z) |Q^z|^2 + \bar{Q}^z \\
    &\Im(Q^z) = \Im(z) |Q^z|^2
  \end{align*}
  Since it is harder, we will only prove the linear concentration of $|Q^z|^2 = (\Im(z)^2 + (\Re(z) - \frac{1}{n}XX^T)^2)^{-1}$. For that we are going to decompose, for any matrix $A \in \mathcal M_{p}$ with unit spectral norm, the random variable $\tr(A|Q^z|^2)$ as the sum of convex and $O(1/\sqrt n)$-Lipschitz mappings of $X$. Let us introduce the two mappings, $\psi : \mathcal M_{p} \to \mathcal M_{p}$ and $\phi : \mathcal M_{p,n} \to \mathcal M_{p}$ defined for any $M \in \mathcal M_{p}$ and $B \in \mathcal M_{p, n}$ with:
  \begin{align*}
    \psi(M) = (\Im(z)^2 + M)^{-1}&
    &\phi(B) = \Re(z)^2 - \frac{2\Re(z)}{n}BB^T + \frac{1}{n^2}BB^TBB^T.
  \end{align*}
  We then have the identity $\tr(A Q^z) = \tr \left(A\psi\circ\phi \left(X\right)\right)$. 

  We then then look at the second derivative of $\psi\circ\phi$ to prove convex properties on $\tr(A\psi\circ\phi)$.
  Given $H \in \mathcal M_{p}$, let us compute:
  \begin{align*}
    \restriction{d\psi}{M} \cdot H = -\phi(M)H\phi(M)&
    &\restriction{d^2\psi}{M} \cdot (H,H) = 2\phi(M)H\phi(M)H\phi(M),
  \end{align*}
  and given $K \in \mathcal M_{p,n}$:
  \begin{align*}
    &\restriction{d\phi}{B} \cdot K = -\frac{2\Re(z)}{n} L(B,K) + \frac{1}{n^2}P(K,B)\\
    &\restriction{d^2\phi}{B} \cdot (K,K) = -\frac{2\Re(z)}{n} KK^T + \frac{2}{n^2}P_2(K,B),
  \end{align*}
  where:
  \begin{itemize}
    \item $L(B,K) = BK^T +KB^T$
    \item $P(B,K) = KB^TBB^T+ BK^TBB^T + BB^TKB^T + BB^TBK^T$
    \item $P_2(B,K) = KK^TBB^T+ KB^TKB^T + KB^TBK^T+ BK^TKB^T + BK^TBK^T + BB^TKK^T$.
  \end{itemize}
  First we deduce from the expression of the first derivative and thanks to Lemma~\ref{lem:Q_borne} that, on $X(\mathcal A)$, $\tr \left(A\psi\circ\phi \right)$ is a $O(\|A\|_F/\sqrt n) = O(1)$-Lipschitz transformation of $X$ (for the Frobenius norm).

  Second,choosing $M = \phi(B)$:
  \begin{align*}
    \restriction{d^2\psi\circ\phi}{B} \cdot (K,K) 
    &= \restriction{d^2\psi}{m
    M} \cdot \left(\restriction{d\phi}{B} \cdot K,\restriction{d\phi}{B} \cdot K\right) + \restriction{d^\psi}{m
    M} \cdot \left(\restriction{d^2\phi}{B} \cdot (K,K)  \right)\\
    &= 2\phi(M) \left(\restriction{d\phi}{B} \cdot K\right)\phi(M) \left(\restriction{d\phi}{B} \cdot K\right)\phi(M) \\
    &\hspace{0.5cm} + \frac{2\Re(z)}{n}\phi(M)  KK^T\phi(M) - \frac{2}{n^2}\phi(M) P_2(K,B) \phi(M).
  \end{align*}
  In this identity the only term raising an issue is $\frac{2}{n^2}\phi(M) P_2(K,B) \phi(M)$ because $P_2(K,B)$ is not nonnegative symmetric. 
  We can however still bound:
  \begin{align*}
    \frac{12}{n^2} \tr \left(A \phi(M) P_2(K,B) \phi(M)\right) \leq \frac{12}{n^2}\|A\|\|\phi(M)\|^2 \|B\|^2 \|K\|_F^2 \leq O \left(\frac{1}{n}\tr(KK^T)\right),
  \end{align*}
  for $B \in X(\mathcal A_Q)$ (in particular $\|B\| \leq O(\sqrt n)$ and $\|\phi(M)\| \leq O(1)$). Now, if we note $h : \mathcal M_{p,n} \to \mathbb R$ defined for any $B \in \mathcal M_{p,n}$ as $h(B) = \frac{1}{n} \tr(BB^T)$, we see that $\frac{1}{n}\tr(KK^T) = d^2h(B)\cdot (K,K)$ is a quadratic functional on $K$, $h$ is thus convex. It is beside $O(1)$-Lipschitz on $X(\mathcal A_Q)$ (for the Frobenius norm). Assuming in a first time that $A$ is nonnegative symmetric and choosing a constant $C \leq O(1)$ sufficiently big, we show that $B \mapsto \tr(A \psi\circ \phi(B)) +C h(B)$ is convex and $O(1)$-Lipschitz on $X(\mathcal A_Q)$ like $C h$. We have thus the concentration:
  \begin{align*}
     \left( \tr(A |Q^z|^2) \ | \ \mathcal A  \right) \in \mathcal E_2 .
   \end{align*} 
   Now, given a general matrix $A \in \mathcal M_{p}$, we decompose $A = A_+ - A_- + A_0$ where $A_+$ and $A_-$ are nonnegative symmetric and $A_0$ is anti-symmetric, in that case $\tr(A |Q^z|^2) = \tr(A_+ |Q^z|^2) - \tr(A_- |Q^z|^2)$, and we can conclude the same way. That eventually gives us the concentration of the proposition.

\bibliographystyle{alpha}

\bibliography{biblio}

\end{document}